\documentclass[a4paper,11pt]{article}
\usepackage{amssymb,amsfonts,amsmath,amsthm,palatino}
\usepackage{hyperref}
\usepackage{floatflt}
\usepackage{graphicx}
\newtheorem{theorem}{Theorem}

\newtheorem{remark}{Remark}
\def\linear{\underset{\mathrm{lin}}{\equiv} }
\def\law{*}

\newcommand{\cubic}{\mathcal{C}}

\newcommand{\quadric}{\mathcal{Q}}
\renewcommand{\line}{\mathcal{L}}

\newcommand{\binomiale}[2]{\left( \begin{array}{c} #1\\#2 \end{array} \right)}

\date{}
\begin{document}

\title{Point-groups over singular cubics}

\author{Emanuele Bellini\thanks{Technology Innovation Institute, UAE}, Nadir Murru\thanks{Universit\'a di Torino, Italy}, Antonio~J. Di Scala\thanks{Politecnico di Torino, Italy}, Michele Elia\thanks{Politecnico di Torino, Italy} 
 }

\maketitle

\thispagestyle{empty}

\begin{abstract}
\noindent
In this paper, we highlight that the point group structure of elliptic curves over finite or infinite fields, may be also observed on singular cubics with a quadratic component. Starting from this, we are able to introduce in a very general way a group's structure over any kind of conics.
In the case of conics over finite fields, we see that the point group is cyclic and lies on the quadric; the straight line component plays a role which may be not explicitly visible in the algebraic description of point composition, but it is indispensable in the geometric description. 
Moreover, some applications to cryptography are described, considering convenient parametrizations of the conics. Finally, we perform an evaluation of the complexity of the operations involved in the parametric groups and consequently in the cryptographic applications.
\end{abstract}

\paragraph{Keywords:} Rational functions, Finite Fields, Public key cryptography, Groups over curves.

\vspace{2mm}
\noindent
{\bf Mathematics Subject Classification (2010): 11F22, 11B50, 11F11}

\vspace{8mm}

\section{Introduction} \label{sect25}
Curves having a group's structure are classical and very important tools in cryptography. The main example is provided by elliptic curves over finite fields, whose use in cryptography was introduced, independently, by Koblitz \cite{Kob} and Miller \cite{Mil}. Moreover, curves with a group's structure, usually cubics or conics, can be exploited for constructing RSA--like schemes (see, e.g., \cite{nadir, Dem, Koy, Murty, Pad, Rao, Sarma, Segar}) for improving the performances in the decryption procedures and having also more security than RSA in some contexts, like broadcast scenarios. Many of these cryptosystems were studied exploiting the properties of the Pell's hyperbola that is the set of solutions in a field $\mathbb F$ of the famous Pell's equation $x^2 - D y^2 =1$, with $D \in \mathbb F^*$. The Pell's hyperbola and its group's structure have been widely studied not only for its cryptographic applications, but also for the natural interest that inspires, see, e.g., \cite{Bar} and \cite{Jac}. 

In this paper, we first focus on the group's structure of the Pell's hyperbola, highlighting its similarity with that of the elliptic curves. Starting from this, we are able to introduce in a very general way a group's structure over any kind of conics (section \ref{sec:group}). Then, we focus on conics defined over finite fields, studying their structure as cyclic groups. This allows to use them in cryptographic applications, especially considering convenient parametrizations (section \ref{sec:order}). Hence, we perform an evaluation of the complexity of the operations involved in the parametric groups, providing also an improvement of a specific algorithm (section \ref{sec:comp}).

\section{Group's structure of conics} \label{sec:group}

In the following, we will refer with the term \emph{product},
in symbols $\otimes$, 
for the operation between two points of a conic and 
we use the term \emph{addition}, 
in symbols $\oplus$, 
for the operation over elliptic curves.

Given two points $\mathbf A = (x,y)$ and $\mathbf B = (w,z)$ of the Pell's hyperbola in the affine plane, their product is obtained as
\begin{equation}
   \label{power}
\mathbf A \otimes \mathbf B = (xw+yz D, xz+yw)  ~~~\Leftrightarrow~~(xw+yz D)+(xz+yw)\sqrt{D}= (x+y\sqrt D)(w+z\sqrt D),
\end{equation}  
that is, from the product of elements in $\mathbb F(\sqrt D)$. This product is usually known as the Brahmagupta product and it can be also introduced in a geometric way \cite{Veb}. 
In fact, let $\mathfrak O=(1,0)$ be a fixed point, the product of two points $\mathbf A$ and $\mathbf B$  is defined as the intersection $\mathbf A \otimes \mathbf B$, with the Pell's hyperbola, of the line through $\mathfrak O$ which is  parallel to the line through $\mathbf A$ and $\mathbf B$. Let us note that $\mathfrak O$ plays the role of the identity.
This construction is the same that gives the operation between points of elliptic curves, even if it appears slightly different, as we will point out below.

Given two points $\mathbf A$ and $\mathbf B$ of an elliptic curve (of equation $y^2 z = x^3 + a x z^2 + b z^3$) in the projective plane, we consider the intersection $\mathbf C$, with the elliptic curve, of the line through $\mathbf A$ and $\mathbf B$. Then $\mathbf A \oplus \mathbf B$ is the symmetric of $\mathbf C$ with respect to the $x$--axis or, in other words, $\mathbf A \oplus \mathbf B$ is the intersection between the elliptic curve with the line through $\mathbf C$ and the identity, which is, in this case, the point at the infinity.

The above construction is the same given on the Pell hyperbola. Indeed, given two points $\mathbf A$ and $\mathbf B$, the line through them intersects the conic in a point $\mathbf C$ that can be only the point at the infinity. 
Then, considering the intersection between the conic with the line through $\mathbf C$ and the identity, which is in this case $\mathfrak O = (0, 1)$, we get $\mathbf A \otimes \mathbf B$. See Figure \ref{fig:1}.
\begin{figure} 
\centering
\resizebox{\textwidth}{!}{
\includegraphics{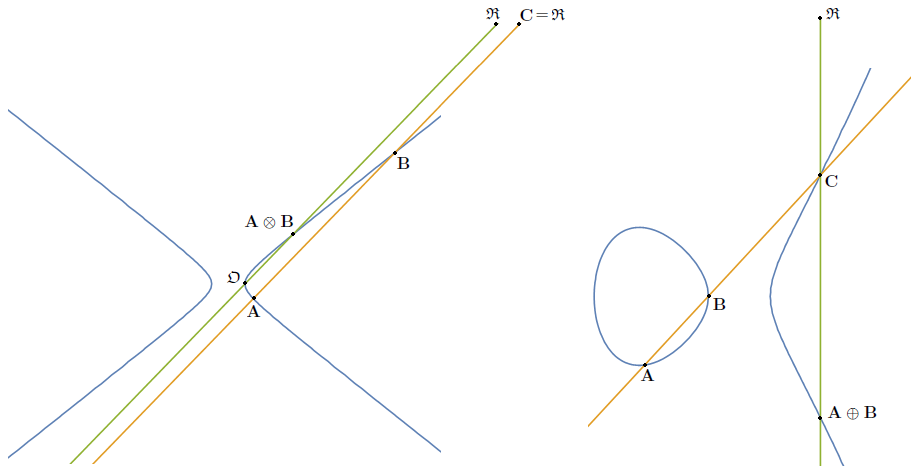}
}
\caption{On the left, the geometric construction of the product between two points of the Pell's hyperbola, where $\mathfrak R$ is the point at the inifnity; on the right the same construction on an elliptic curve.} 
\label{fig:1}
\end{figure}

As a matter of fact, this hyperbola point-group structure is simply another way to see the operation over a degenerated cubic with two components. Therefore, we can consider a more general situation where Pell's hyperbola is substituted by any quadric, i.e. hyperbola, ellipsis, or parabola, furthermore, as we will see, the identity point can be any point of the conic. 

Let $\cubic[\mathbb F]$ be a cubic having two components: a straight line $\line[\mathbb F]$ of equation $ax+by+c=0$,  and a non-degenerated quadric $\quadric[\mathbb F]$ of equation
$$  e x^2 + 2gxy+fy^2+dx+hy+k=0 ~~. $$
Fix a point $\mathfrak O=(\alpha,\beta)$ on $\quadric[\mathbb F]$, then the product of two points
$\mathbf A=(x,y)$ and $\mathbf B=(w,z)$ on the quadric is defined to be the point $\mathbf P=\mathbf A \otimes \mathbf B$ obtained as the intersection of the line through $\mathfrak O=(\alpha,\beta)$ and $\mathbf S$ which is given by the intersection of $\cubic[\mathbb F]$ with the line through $\mathbf A$ and $\mathbf B$. 
The geometric view immediately shows that the operation $\otimes$ is commutative with $\mathfrak O$ as the identity, as well as the existence of inverses. Indeed, the inverse of a point $\mathbf A$ can be obtained considering, firstly, the intersection $\mathbf T$ between the line tangent to the conic through $\mathfrak O$ and $\ell$, then $\mathbf A^{-1}$ is the intersection between the conic and the line through $\mathbf A$ and $\mathbf T$.
The associativity can be proved geometrically or algebraically. In this second instance, it is convenient
to refer to quadrics in canonical, or reduced, form:
$$      \begin{array}{lcl}
                 x^2 -Dy^2=\ell  &~~& \mbox{for the hyperbolas and ellipse, i.e., $g^2-ef \neq 0$,}  \\
                 y= e x^2 + k  &~~& \mbox{for the parabola, i.e. $g^2-ef = 0$.} 
          \end{array}
$$
The proof is indirect by showing first that we have a parametrization, and thus showing that the set of parameters admits of a group structure. 

\paragraph{Case $g^2-ef \neq 0$.} It is necessary and convenient to distinguish two further cases.

When $\ell=u^2$, let $m$ be a parameter, and rewrite the equation  $x^2 -Dy^2=u^2$ as follows
$$    
\frac{y}{x-u} = \frac{x+u}{Dy} =m   
$$
thus solving for $x$ and $y$ we have
$$  
x = 
\frac{u(D m^2+1)}{D m^2-1} ~~,~~ 
y = 
\frac{2 u m}{D m^2-1} ~~.   
$$
The product of the two points 
$\mathbf A=(x(m_A),y(m_A))$ and 
$\mathbf B=(x(m_B),y(m_B))$ 
is the new point 
$\mathbf P=(x(m_P),y(m_P))$ where
\begin{equation}
    \label{equ2}
   m_P = m_A \odot m_B =\frac{D m_A m_B+1}{(m_A+m_B)D} ~~. 
\end{equation}
Using this expression, the parameter characterizing the sum of three points $\mathbf A, \mathbf B, \mathbf C$ is
$$    \frac{D m_A m_B m_C+m_A+m_B+m_C}{(m_A m_B+m_B m_C +m_C m_A)D+1}  ~~,   $$
the symmetry proves that the product of points is associative.

When $\ell\neq u^2$, let $m$ be a parameter and $\mathfrak O =(\alpha,  \beta)$ be the fixed point on the conic
(which plays the role of group identity). Consider the line through $\mathfrak O$ with slope $m$,
that is of equation $y-\beta= m(x-\alpha)$, 
which meets the hyperbola in a second point $\mathbf B$ of coordinates
$$   
x=\alpha+ \frac{2(- \beta D m+\alpha)}{d m^2-1}   ~~,~~ y=\beta+ m \frac{2(-\beta D m+\alpha)}{D m^2-1}  ~~.    
$$
The product of the two points 
$\mathbf A=(x(m_A),y(m_A))$ and 
$\mathbf B=(x(m_B),y(m_B))$ is the point 
$\mathbf P=(x(m_P),y(m_P))$ where
\begin{equation}
    \label{eqnu2}
    m_P = m_A \odot m_B = \frac{(D m_A m_B+1) \alpha-(m_A+m_B) \beta D}{(-(D m_A m_B+1) \beta+(m_A+m_B) \alpha) D} ~~. 
\end{equation}
Again, the symmetry occurring in the product of three points proves that the product of points is associative.

\paragraph{Case $g^2-ef = 0$.}  Let $m$ be a parameter and $\mathfrak O =(\alpha,  \beta)$ be the fixed point on the parabola, the second intersection of the line through $\mathfrak O$ is a point of coordinates
$$  
x = \frac{-\alpha e+m}{e}  ~~,~~ y =\frac{\alpha^2 e^2-2\alpha e m+e k+m^2}{e}    ~~.  
$$
The product of two points 
$\mathbf A=(x(m_A),y(m_A))$ and 
$\mathbf B=(x(m_B),y(m_B))$ 
is the point 
$\mathbf P=(x(m_P),y(m_P))$ where
\begin{equation}
    \label{equ0}
    m_P = m_A \odot m_B = -2 \alpha e+m_A+m_B.
\end{equation}
Again, the symmetry occurring in the product of three points proves that the product of points is associative.

We conclude this section observing that it is a well-known fact for complex algebraic geometers that smooth curves with a law group have genus $g=1$.
This is so because the law group allows to define a non vanishing vector field (just by left translations), i.e.,
the tangent bundle is trivial.
Thus the Euler characteristic $\chi = 2 - 2 g$ vanishes, 
hence $g=1$.
Any conic has genus $0$, so ones wonder what is going wrong. The point is that in this case our conics are affine curves.
That is to say, the argument with the Euler characteristic works for compact curves (so called projective curves).

It is also a common practice to introduce the law group on a plane cubic by using the classic secant-tangent
construction. That is to say, on smooth projective cubic we take a flex $\mathfrak R$ as neutral element and introduce the group law $\law$ by using the secant-tangent construction i.e. by using the linear series of divisors $|3 \mathfrak R|$:
\[ \mathbf{P} + \mathbf{Q} + \mathbf{P}^{-1} \law \mathbf{Q}^{-1} \linear 3 \mathfrak R  \]
The converse of this procedure is not widely known. Namely,

\begin{theorem} 
Let $\mathcal C$ be a complete smooth curve which is also an algebraic group $(\mathcal C,\law)$.
Then $\mathcal C$ has genus $g=1$ and $(\mathcal C,\law)$ is abelian. Moreover, for arbitrary $\mathbf{P}, \mathbf{Q} \in \mathcal{C}$:

\[ \mathbf{P} + \mathbf{Q} + \mathbf{P}^{-1} \law \mathbf{Q}^{-1} \linear 3 \mathfrak R  \]

where $\mathfrak R$ is the neutral element of $\law$ and $\mathbf{P}^{-1}$ the inverse.
Thus $3 \mathfrak R$ is a very ample divisor and the law group $\law$ comes
from the classical secant-tangent construction by
using the embedding into $\mathbb{P}^2$ given by the complete linear series $|3 \mathfrak R|$.
\end{theorem}
\begin{proof}
Here we intend a complete curve as in \cite[Chapter IV]{Har}.
Denote by $\mathbf{P} \law \mathbf{Q}$ the law operation and by $\mathbf{P}^{-1}$ the inversion of $\mathbf{P}$.
As we explained above, $g = 1$, 
because using left multiplication 
we get a never vanishing differential
$\mathrm{d}z$, i.e. 
the canonical bundle=cotangent bundle is trivial. 
So any other differential, say $\alpha,$ is a multiple of $\mathrm{d}z$ : 
$$
\alpha = f\mathrm{d}z
$$
where $f$ is a regular function hence constant since $\mathcal C$ is complete. Thus $g=1$.

That $\law$ is abelian is a special case of \cite[Theorem 2]{De}.
By putting $\mathbf{D} = \mathbf{P} \law \mathbf{Q}$ in \cite[Theorem 5]{De}
we get \[ \mathbf{P} + \mathbf{Q} \underset{\text{lin}}{\equiv}  \mathfrak R + \mathbf{P} \law \mathbf{Q} \]
where $\mathfrak R$ is the neutral element of $\law$. Since this is true for arbitrary points we get:
\[ \mathbf{P} \law \mathbf{Q} + \mathbf{P}^{-1} \law \mathbf{Q}^{-1}  \linear 2 \mathfrak R\, \]
hence for arbitrary $\mathbf{P}, \mathbf{Q}$ we get \[ \mathbf{P} + \mathbf{Q} + \mathbf{P}^{-1} \law \mathbf{Q}^{-1} \linear 3 \mathfrak R  \]

So the linear series $|3\mathfrak R|$ is a complete $g^2_3$, i.e. $2 = \mathrm{dim}|3\mathfrak R|, 3=\mathrm{deg}(\mathbf{D})$,
and it is very ample as follows from \cite[pag. 307]{Har}.
Then $\mathcal{C}$ is embedded in $\mathbb{P}^2$ as a cubic by $|3\mathfrak R|$. The equation \[ \mathbf{P} + \mathbf{Q} + \mathbf{P}^{-1} \law \mathbf{Q}^{-1} \linear 3 \mathfrak R  \]
tells us that the law group $\law$ is given by the secant-tangent construction.
\end{proof}

\begin{remark}
The binary operation $*$ on the conic can be extended to all pairs ${\mathbf{P},\mathbf{Q}}$ of the
singular projective cubic with exclusion of the pair $\{ \mathbf S_1, \mathbf S_2 \}$ of two singular points intersection of the conic with the line at infinity.
Here is the explicit formula 
\[ [x:y:u] * [w:z:v] = [xw + yzD : xz+yw : uv ]. \]
\end{remark}

\section{Group and group order} \label{sec:order}
When $\mathbb F$ is a finite field, the point product defines a finite group on the quadrics which depends on a single parameter, thus, it
is expected that these groups are cyclic. This is the case is proved along with the determination of group order.  

Let $q=p^m$ be an odd prime power, and consider the curve $x^2-Dy^2-z^2=0$ in the projective plane. The points at infinity have
coordinates $(\pm\sqrt D,1,0)$, where  $\sqrt D$ belongs to $\mathbb F_q$ if $D$ is a square
in the field, otherwise it belongs to the extension field $\mathbb F_{q^2}$, i.e., it is not an 
$\mathbb F_q$-point. Considering the Pell's equation written as $Dy^2=x^2-1$, and rising both sides
to the power exponent $\frac{q-1}{2}$ we have
$$   (x^2-1)^{\frac{q-1}{2}}= \left\{ \begin{array}{l}
       1 ~~~\mbox{if $D$ is a square and $y \neq 0$}  \\
       -1~~~\mbox{if $D$ is not a square and $y \neq 0$}
     \end{array}  \right.
$$
therefore:
\begin{itemize}
   \item $D$ not square: since $\left( (x^2-1)^{\frac{q-1}{2}}\right)^{q+1}=1$ if $x^2-1 \neq 0$, the total number of points on the curve is $q+1$ including the two points $(\pm 1,0)$
   \item $D$ square: since $(x^2-1)^{\frac{q-1}{2}}=1$ if $x^2-1 \neq 0$, the total number of points is
     $q-1$  including the two points $(\pm 1,0)$, but in this case we must also count the two points
   at infinity, that in projective coordinate are  $(\pm \sqrt D,1,0)$, hence in total the group has still order $q+1$.
\end{itemize}

\noindent
This count also holds for any square $\ell$, in which case the coordinates $x$ and $y$ are changed by the same factor $u$, with $u^2=\ell$.

\noindent
If $\ell$ is not a square in the field, the set of solutions can be obtained as the product 
$$ (u+ \sqrt{D} v)  (x_o+ \sqrt{D} y_o)    $$
where   $u+ \sqrt{D} v$ is any solution of the Pell equation and $x_o+ \sqrt{D} y_o$ is a fixed solution
 of  the equation $ x^2 -Dy^2=\ell$.

The group of the parabola of equation $ y= e x^2 + k$ is cyclic of order $q+1$. Clearly the equation identifies $q$ points, the further point is the point at infinity, 
which is characterized by the homogeneous equation 
$ y z = e x^2 + k z^2$.
By setting $z=0$, we obtain $x=0$, 
thus the point at infinity has homogeneous coordinates 
$(0,1,0)$.

In many algorithms for cryptographic applications, the use of the arithmetic of algebraic curves
typically requires the evaluation of a multiple $n \mathbf{A}$ or a power $\mathbf{A}^n$ of a point on the curve with large $n$.
The use of quadrics does not avoid this computation, however the computational cost may be significantly smaller, although maintaining the same strength against cryptanalytic attacks. 
Furthermore, the point product is a complete operation, 
which means  that the formulas 
are defined for all pairs of input points on the quadric, with 
no exceptions for doubling, for the neutral element, or for negatives, 
and the output is always a point on the curve
\cite{bernstein2007faster}.
In particular, in the case of quadrics, it appears convenient to perform the operations on the corresponding set of parameters, like, e.g., in \cite{nadir, Pad}, where the authors only focused on the Pell hyperbola, while here we have showed that it is possible to work on more general conics. We give a sketch of the RSA-like cryptosystem on general conics.

\begin{itemize}

\item \textbf{Key generation.}
\begin{itemize}
\item Take two large prime numbers $p$ and $q$ and evaluate$N = p q$
\item Take an integer $\epsilon$ coprime with $(p+1)(q+1)$ and evaluate $\delta \equiv \epsilon^{-1} \pmod{(p+1)(q+1)}$
\end{itemize}
The couple $(N, \epsilon)$ is the public key, the triple $(p, q, \delta)$ is the private key.

\item \textbf{Encryption.}\\
Let $M_x, M_y \in \mathbb Z_N^*$ be two plaintexts.
\begin{itemize}
\item Find the conic where $(M_x, M_y)$ lies (for instance, in the case of the Pell hyperbola, evaluate $D = (M_x^2 - 1) / M_y^2$ in $\mathbb Z_N$; $D$ identifies the Pell hyperbola where the point lies)
\item Find the parameter $m$ corresponding to $(M_x, M_y)$ (for instance, in the case of the Pell hyperbola $m = (1 + M_x) / M_y$ in $\mathbb Z_N$)
\item Evaluate $c = m^{\odot \epsilon} \pmod N$, where powers are evaluated with respect to $\odot$ that is the operation over the set of parameters (i.e., the operation described by equations \eqref{equ2}, \eqref{eqnu2}, \eqref{equ0}, depending on the type of conic).
\end{itemize}
The encrypted message is $c$.

\item \textbf{Decryption.}\\
Let $c$ be the ciphertext.
\begin{itemize}
\item Evaluate $c^{\odot \delta} \pmod N$ that returns $m$
\item Find the point corresponding to the parameter $m$, i.e., the plaintexts $(M_x, M_y)$ (for instance, in the case of the Pell hyperbola evaluate $(m^2 + D)/(m^2 - D)$ and $2 m / (m^2 - D)$ in $\mathbb Z_N$).
\end{itemize}

\end{itemize}

\begin{remark}
If we construct the previous cryptosystem using a parabola of equation $y = e x^2 + k$, with identity $(\alpha, \beta)$, then the cryptosystem appears to be weak, since the exponentiation with respect to the operation described by \eqref{equ0} has a closed form:
\[ m^{\odot \epsilon} = -(2 \epsilon -2) \alpha e + \epsilon \pmod N, \]
where $\epsilon, \alpha, e$ are known quantities. However, previously, we have studied the group's structure also for the parabola for the completeness of the discussion.
\end{remark}

\section{Complexity of the computations} \label{sec:comp}

In this section, we evaluate the complexity of the operations involved in cryptosystems constructed on conics, referring to the parametric representation and considering the equations (\ref{equ2}), (\ref{eqnu2}), and (\ref{equ0}). We study also the case given by the parabola just for completeness. The complexity is expressed
 in terms of number of arithmetical operations in $\mathbf F_q$, i.e., number of multiplications, additions, and divisions, or inversions, being well known that inversion in finite fields is an expensive operation. In particular, we will focus the attention on the operation over hyperbolas, since the main cryptosystems in the literature are developed on these curves. We will present a direct method of evaluation of the exponentiation over the set of the parameters for all the conics (subsection \ref{sub:direct}). Moreover, for the exponentiation over the set of parameters of hyperbolas, we evaluate the complexity of More's algorithm (subsection \ref{sub:more}) and of an improvement of it that we propose in subsection \ref{sub:morep}.

In the following, let 
$n=\sum_{i=0}^\ell b_i2^i$, 
with $b_i \in \{0,1\} \subset \mathbb N$, and 
$\ell=\lfloor \log_2 n\rfloor,$ 
be the binary representation of $n$, and 
denote with $w(n)=\sum_{i=0}^\ell b_i$ the number of symbols $1$ in the binary representation of $n$.
Three different algorithms, all exploiting the square-and-multiply method, for computing $\mathbf{A}^{\otimes n}$, with $\mathbf A$ a point of a conic, are described and compared. The computation scheme is the following 

\begin{itemize}
   \item[i)] Find the parameter value $m$ of $\mathbf{A}$
   \item[ii)] Compute $m^{\odot n}$ that is the parameter of $\mathbf{A}^{\otimes n}$
   \item[iii)] Find the coordinates of the resulting point $\mathbf{A}^{\otimes n}$
\end{itemize}

\noindent
Since the computations at steps i) and iii) are common to every method, they are not counted in this complexity evaluation. The product of a point with itself is called doubling.

\subsection{Direct Algorithm} \label{sub:direct}

The procedure computes and stores $\ell$ doublings, denoted by $x_j$, $j=1, \ldots , \ell$, of the initial $m$ using the equations (\ref{equ2}), (\ref{eqnu2}), and (\ref{equ0}), then iteratively evaluates the
 sum $\sum_{j=0}^\ell b_j x_j$ by means of the same equations.
For doing these computations it is convenient to write the equations (\ref{equ2}) and (\ref{eqnu2})
in a slightly different form
\begin{equation}
    \label{equ2bis}
   m_A \odot m_B =\frac{ m_A m_B+D^{-1}}{m_A+m_B} ~~,~~~~m^{\odot 2} =\frac{ m^2+D^{-1}}{m+m}  ~~,
\end{equation}
and
\begin{equation}
    \label{eqnu2bis}
    m_A \odot m_B = \frac{(m_A m_B+D^{-1})-(m_A+m_B)\frac{\beta}{\alpha}}{-( m_A m_B+D^{-1}) D\frac{\beta}{\alpha}+(m_A+m_B)} ~~, ~~m^{\odot 2}= \frac{(m^2+D^{-1})-(m+m)\frac{\beta}{\alpha}}{-( m^2+D^{-1}) D\frac{\beta}{\alpha}+(m+m)} ~~,
\end{equation}
respectively.
Once the sequence $x_j$ has been computed, the result is obtained by the following iterated sums

\begin{itemize}
   \item If $b_0 = 1$, set $y_0 = m$, otherwise $y_0=\infty$
   \item For $j$ from $1$ to $s$ do 
$$  \left\{ \begin{array}{lcl}    
          y_j=y_{j-1}\odot x_j  &~& \mbox{if ~ $b_j \neq 0$}  \\
          y_j=y_{j-1} & & otherwise  \\
        \end{array}  \right.   
$$
   \item Output $m^{\odot n}=y_s$
\end{itemize}

\subsection{More's algorithm} \label{sub:more}

The computation into the group of parameters may be performed using R\'edei polynomials, when the group operation is defined by the equation (\ref{equ2bis}), see \cite{lidl} for an overview on Rédei polynomials and see \cite{nadir2, nadir} for the connection with the product \eqref{equ2bis} that we will recall below. 

R\'edei polynomials are defined as follows:
$$\mathcal N_n(D, z) = \sum_{k=0}^{\lceil \frac{n}{2} \rceil} \binomiale{n}{2k} D^{k}z^{n-2k}~~,~~
    \mathcal D_n(D, z) = \sum_{k=0}^{\lceil \frac{n}{2} \rceil} \binomiale{n}{2k+1} D^{k}z^{n-2k-1} ~~, $$
and satisfy the linear recurrences
$$      \left\{  \begin{array}{lcl}
               \mathcal N_{n+1}(D, z) & = & z \mathcal N_n(D, z) + d \mathcal D_n(D, z)\\
               \mathcal D_{n+1}(D, z) & = &  \mathcal N_n(D, z) +  z \mathcal D_n(D, z)  \\
          \end{array}  \right.  
$$
or, equivalently, satisfy the homogeneous linear recurrence of order two
$$      x_{n+2}-2zx_{n+1}+(z^2-D) x_n=0~~, $$
with respective initial conditions 
$$ \left\{  \begin{array}{lcl} 
                x_0=\mathcal N_0(D,z)=z   &~,~&  x_1=\mathcal N_1(D,z)=z^2+D \\ 
                x_0= \mathcal D_0(D,z)=1   &~,~&  x_1=\mathcal D_1(D,z)=2z
              \end{array}   \right.
$$
The R\'edei rational function is the ratio  $Q_n(D,z)= \frac{\mathcal N_n(D, z)}{D_n(D, z)}$ and gives $z^{\odot n}$ (that is the powers of a parameter corresponding to a point of the Pell's hyperbola $x^2 - D y^2 = 1$). Both R\'edei polynomials and rational functions can be quickly evaluated via their recurrence relations.
Many properties can be deduced from the relation
$$  \left[ \begin{array}{ll}
                  \mathcal N_n(D, z)  &  D \mathcal D_n(D, z)   \\
                 \mathcal D_n(D, z)  &  \mathcal N_n(D, z)
           \end{array}  \right]  =
\left[ \begin{array}{ll}
                   z  &  D   \\
                   1  &  z
           \end{array}  \right]^n
$$
From this representation of R\'edei polynomials we immediately have the following relations for the  R\'edei rational functions
$$   Q_{n+m}(D,z) = Q_{n}(D,z) \odot Q_{m}(D,z)  ~~,~~ Q_{nm}(D,z)=Q_{n}(D,Q_{m}(D,z)) ~~.  $$

In \cite{more}, More proposed a fast algorithm for evaluating the R\'edei rational functions.
 Though, the algorithm uses $2$ inversions at each step. Precisely, in $\mathbb F_{q}$, the number of multiplications required for computing the inverse of an element is $O(\log_2 q)$. Therefore, the actual complexity of More's algorithm is $O(\log_2 n \cdot \log_2 q)$.
However, as shown below, the algorithm can be modified to avoid inversions at each step by using more multiplications, and using only one inversion before returning the result.

More's algorithm mimics the square-and-multiply algorithm for evaluating powers, and evaluates the R\'edei function $Q_n(x)$ of degree $n \geq 1$ with respect to $t(x)=x^2- a x- b$ (i.e., for a more general definition of Rédei functions). 
It consists of three steps:
\begin{itemize} 
   \item[A1] 
        Let $[b_0, \ldots, b_{\ell-1}]$ be the binary representation of $n \geq 1$ with $b_k \in \{0, 1\}$. \\
        Set $i \leftarrow \ell-1$ and $R(x) \leftarrow x$, compute $x+a$.
  \item[A2]   If $i < 0$, the algorithm terminates, with $R(x)$ as the answer.
  \item[A3]   Set $ R(x) \leftarrow \frac{R^2(x)+b}{2R(x)+a}$ 
 \begin{enumerate}
    \item[] If $b_i= 1$,  set  $R(x) \leftarrow \frac{xR(x)+b}{R(x)+x+a}$
    \item[] Set $i \leftarrow i-1$ and return to step A2. 
  \end{enumerate} 
\end{itemize}

The output of this procedure is $R(x)$ that is the Rédei rational function $Q_n(x)$, which coincides with $x^{\odot n}$ when the polynomial is $t(x) = x^2 - D$.
Noting that expression $2R(x)+a$ can be evaluated as $R(x)+R(x)+a$, each step requires 
 $1$ multiplication (i.e. $R^2(x)$), $2$ additions, and one division if $b_i=0$, while if  $b_i=1$ one further
multiplication, $2$ further additions, and one further division are required. Note that a division can be
done using one inversion and one multiplication.
In summary, the algorithm requires ~$w(n)+\ell$~ multiplications, this same number of inversions, and $2(w(n)+\ell)$ additions. 

In the following section an algorithm which uses only one inversion is described and its complexity estimated.

\subsection{Modified  More's algorithm} \label{sub:morep}

A way to avoid the division at every step is to consider $R(x)$, in the More's algorithm, as the ratio of two polynomials $\frac{A(x)}{B(x)}$ so that we can update the polynomials $A(x)$ and $B(x)$ at each step and only at the end perform the quotient.

\begin{itemize} 
   \item[A1] 
        Let $[b_0, \ldots, b_{\ell-1}]$ be the binary representation of $n \geq 1$ with $b_k \in \{0, 1\}$. \\
        Set $i \leftarrow \ell-1$, $A(x) \leftarrow x$ and $B(x) \leftarrow 1$, compute $x+a$.
  \item[A2]   If $i < 0$, the algorithm terminates, with $R(x) \leftarrow A(x) / B(x)$ as the answer.
  \item[A3]   Set $ A(x) \leftarrow A(x)^2 + b B(x)^2$ and $B(x) \leftarrow 2 A(x) + a B(x)$ 
 \begin{enumerate}
    \item[] If $b_i= 1$,  set  $A(x) \leftarrow x A(x) + b B(x)$ and $B(x) \leftarrow 2 A(x) + (x+a) B(x)$
    \item[] Set $i \leftarrow i-1$ and return to step A2. 
  \end{enumerate} 
\end{itemize}

Also in this case, the output is $R(x)$ that is the Rédei rational function $Q_n(x)$, which coincides with $x^{\odot n}$ when the polynomial is $t(x) = x^2 - D$. Let us note that, in the iterations, no inversion is needed, and only a final quotient is computed to provide the value of the R\'edei rational function.

This modified algorithm needs some additional multiplication and sum at each step, that is 
   \begin{enumerate}
      \item $5$ multiplications, $3$ additions, and zero inversion for $\ell$ steps:
      \item $3$ multiplications,  $3$ additions, and zero inversion for $w(n)$ steps.
  \end{enumerate}

In summary the total number of multiplications is $5\ell+3w(n)$, the number of additions is 
    $3\ell+3w(n)$, plus a final division.

\noindent
For the sake of comparison the complexities relative to the three algorithms are summarized in Table \ref{tab:1}. 
Note that an inversion costs about $\log_2 q$ multiplications.

\begin{table}[h]
\begin{center}
\begin{tabular}{|c |c|c|c|| c|c|c|| c|c|c|}  \hline
                &\multicolumn{3}{|c||}{Direct }&   \multicolumn{3}{c||}{More}     & \multicolumn{3}{c|}{Modified More} \\    \hline
                       &  P             & A              & I             &  P    & A  & I & P & A & I   \\    \hline
 (\ref{equ2})    &$2\ell+2w$&$2\ell+3w$&$\ell+w$ & $2\ell$ & $3\ell$ &$\ell$ &  $5\ell+3w$  & $3\ell+3w$ & $1$ \\  \hline
 (\ref{eqnu2})  &$4\ell+4w$&$4\ell+3w$&$\ell+w$ & -- &--  & --& --&-- & -- \\  \hline
 (\ref{equ0})      & --   & $2\ell+2w$  & -- &--    &-- &-- & --& --&-- \\  \hline
\end{tabular} 
\vspace{4mm}

P = \# products ~~~~A= \# additions ~~~~ I = \# inversions
\end{center}    
    \caption{Complexity of three algorithms to compute the exponentiation $m^{\odot n}=y_s$.}
    \label{tab:1}
\end{table}

\section*{Acknowledgment}
A.~J. Di Scala is member of GNSAGA of INdAM and of DISMA Dipartimento di Eccellenza MIUR 2018-2022.

\end{document}